\documentclass[oneside]{amsart}
\usepackage{amssymb}

\numberwithin{equation}{section}
\newcommand{\be}{\begin{equation}}
\newcommand{\ee}{\end{equation}}

\newcommand{\R}{\mathbb R}
\newcommand{\pd}{\partial}
\newcommand{\ep}{\varepsilon}

\renewcommand{\phi}{\varphi}

\newtheorem{theorem}{Theorem}[section]

\newtheorem{corollary}[theorem]{Corollary}
\newtheorem{lemma}[theorem]{Lemma}

\theoremstyle{definition}

\theoremstyle{definition}

\newtheorem{df}[theorem]{Definition}
\newtheorem*{notation}{Notation}

\theoremstyle{remark}
\newtheorem{remark}[theorem]{Remark}
\newtheorem*{remark*}{Remark}

\DeclareMathOperator{\len}{length}
\DeclareMathOperator{\dist}{dist}

\begin{document}

\renewcommand{\L}{\mathcal L}

\title[On intrinsic geometry of surfaces in normed spaces]
{On intrinsic geometry of surfaces \\ in normed spaces}

\author{Dmitri Burago}                                                          
\address{Dmitri Burago: Pennsylvania State University,                          
Department of Mathematics,
University Park, PA 16802, USA}                      
\email{burago@math.psu.edu}                                                     
                                                                                
\author{Sergei Ivanov}
\address{Sergei Ivanov:
St.~Petersburg Department of Steklov Mathematical Institute,
Fontanka 27, St.Petersburg 191023, Russia}
\email{svivanov@pdmi.ras.ru}

\thanks{The first author was partially supported                                
by NSF grants DMS-0604113 and DMS-0412166.                                      
The second author was partially supported by the                                
Dynasty foundation and RFBR grant 08-01-00079-a.}

\begin{abstract}
We prove three facts about intrinsic geometry of surfaces in a normed 
(Minkowski) space. When put together, these facts
demonstrate a rather intriguing picture.
We show that (1) geodesics on saddle surfaces
(in a space of any dimension) behave as they are expected to:
they have no conjugate points and thus minimize length in their homotopy class;
(2) in contrast, every two-dimensional Finsler manifold can be
locally embedded as a saddle surface in a 4-dimensional space;
and (3) geodesics on convex surfaces in a 3-dimensional space also
behave as they are expected to: on a complete strictly convex surface,
no complete geodesic minimizes the length globally.
\end{abstract}



\maketitle

\section{Introduction}

The goal of this paper is to prove three facts about intrinsic
geometry of surfaces in a normed (Minkowski) space.
When put together, these facts
demonstrate a rather intriguing picture. Namely, Theorem \ref{t-saddle-surface}
asserts that geodesics on saddle surfaces
(in a space of any dimension) behave as they are expected to:
they have no conjugate points and thus minimize length in their homotopy class.
In contrast, Theorem \ref{t-saddle-embedding} says that every 
two-dimensional Finsler manifold can be locally embedded as a saddle surface
in a 4-dimensional normed space.

Thus the fact that geodesics on saddle surfaces minimize the length is global and, 
unlike in Riemannian geometry, it cannot be derived from studying local invariants
such as the Gaussian curvature.
Note that the property that a surface is saddle has
nothing to do with various types of Finsler curvatures,
for they can be negative or positive at
some points of cylindrical surfaces.

Furthermore, Theorem \ref{t-convex-surface} 
asserts that geodesics on convex surfaces 
(in a 3-dimensional space) also
behave as they are expected to:
on a complete strictly convex surface,
no complete geodesic minimizes the length globally
(and therefore some geodesics have conjugate points.)
Therefore such a surface cannot be re-embedded as a
saddle surface in any normed space
(even though it can be re-embedded locally, hence this
obstruction is of global nature). The nature of these phenomena remains obscure 
to us. 

\begin{remark*}
Interestingly enough, for \textit{polyhedral} surfaces
in normed spaces, global minimality of geodesics
can be deduced from local intrinsic geometry: a globalization theorem holds.
Studying Finsler geodesics has nice applications where there is no word
``Finsler'' in the formulation. For instance, consider a braid of several 
strings connecting two sets of nails in two parallel planes in $\R^3$. Having 
fixed topological type of the braid, one asks if the braid with the shortest 
total length of strings is unique (and if so, how convex is the length function
near the optimum, compare with \cite{BGS}).
This question, having started from a purely Euclidean setup,
naturally reduces to a problem about Finslerian geodesics.
(We are grateful to Rahul \cite{braid-question} who brought this
question to our attention.)
We will address this aspect of geometry of polyhedral Finsler manifolds along with
a few others elsewhere. 
\end{remark*}

Now we proceed to definitions and formulations.
Let $\|\cdot\|$ be a norm on a finite dimensional
vector space $V$. Note that the norm is
uniquely determined by its unit ball
$B=\{v\in V:\|v\|\le 1\}$ which is a
centrally symmetric convex body in~$V$.
The boundary of $B$ is the unit sphere
of $\|\cdot\|$, it also determines the norm uniquely.

We say that a norm is $C^r$-smooth if
it is a $C^r$ function on $V$ away from the origin.
This is equivalent to the property that
the unit sphere of the norm is a
$C^r$ hypersurface in~$V$.
If the $C^r$ prefix is omitted, the term ``smooth'' means $C^\infty$
(though the results are probably valid for $C^2$, we just did not care 
to chase the number of derivatives through the proofs).

A norm $\|\cdot\|$ is said to be \textit{strictly convex}
if its unit sphere does not contain straight line segments.
This is equivalent to the property that the triangle inequality
$$
 \|v+w\| \le \|v\| + \|w\|, \qquad v,w\in V
$$
is strict unless $v$ and $w$ are proportional.

A norm $\|\cdot\|$ on $V$ is said to be \textit{quadratically convex}
if for every $v\in V\setminus\{0\}$ there is a positive definite
quadratic form on $V$ whose square root majorizes the norm
everywhere and equals the norm on the vector~$v$. For smooth norms,
this is equivalent to the following: the function $\|\cdot\|$
has positive definite second derivative at every point of $V\setminus\{0\}$.
Smooth quadratically convex norms are called \textit{Minkowski norms}.

A (reversible) \textit{Finsler metric} on a smooth manifold $M$ is a continuous
map $\phi\colon TM\to\R$ which is smooth away from the zero section
and such that for every $x\in M$ the restriction of $\phi$ to
$T_xM$ is a Minkowski norm.
A \textit{Finsler manifold} is a manifold $M$ equipped
with a Finsler metric.
A detailed treatise of differential geometry of Finsler
manifolds can be found e.g.\ in \cite{BCS},
below is a list of the basic definitions and facts that we use.

The length of a smooth curve $\gamma\colon[a,b]\to M$
in a Finsler manifold $M=(M,\phi)$ is defined by
$$
 \len(\gamma) = \int_a^b \phi(\dot\gamma(t))\,dt .
$$
Geodesics in $M$ are locally length minimizing curves.
Equivalently, geodesics are critical points of the
energy functional $\gamma\mapsto\int\phi^2(\dot\gamma)$,
they are determined by the corresponding Euler--Lagrange equation.
Smoothness and quadratic convexity of~$\phi$ ensure that
this equation is non-degenerate and imply the usual
existence and uniqueness properties of solutions.
All geodesics in this paper are assumed parameterized
by arc length.

Surfaces in normed spaces are natural examples of
Finsler manifolds.
Namely if $V$ is a vector space with a Minkowski norm $\|\cdot\|$
and $M$ is a smooth manifold, then every smooth immersion
$f\colon M\to V$ induces a Finsler $\phi$ metric on $M$
given by $\phi(v)=\|df(v)\|$ for all $v\in TM$.
If $\phi$, $f$ and $\|\cdot\|$ are so related,
one also says that $f$ is an \textit{isometric immersion}
of $(M,\phi)$ to $(V,\|\cdot\|)$.

\begin{df}
A two-dimensional smooth surface $S$ in $\R^n$
(that is, a smooth immersion $S\colon M\to\R^n$
where $M$ is a two-dimensional manifold)
is \textit{strictly saddle} (resp.\ \textit{saddle})
at a point $p\in M$
if, for every normal vector at $p$, the second fundamental form
of $S$ with respect to this normal vector is indefinite
(resp.\ indefinite or degenerate). A surface is (strictly)
saddle if it is (strictly) saddle at every point.
\end{df}

One easily sees that this definition is affine invariant
(or, equivalently, is independent of the Euclidean structure
in the ambient space). Therefore is makes sense for surfaces
in a vector space (without any Euclidean structure).
In a Euclidean space, saddle surfaces have non-positive
Gaussian curvature and therefore their geodesics
have no conjugate points.
Furthermore, only saddle surfaces preserve non-positiveness
of curvature under all affine transformations, cf.~\cite{Shefel}.

The main result of this paper is the following theorem
asserting that the ``no conjugate points'' property of
saddle surfaces holds true in non-Euclidean normed spaces
as well.

\begin{theorem}
\label{t-saddle-surface}
Let $V$ be a finite dimensional space
with a Minkowski norm
and $S$ a smooth saddle surface in $V$.
Then every geodesic segment on $S$
minimizes the length among all $C^0$-nearby curves with
the same endpoints.
\end{theorem}

The standard argument (similar to the proof of the
Cartan--Hadamard theorem) shows that
Theorem \ref{t-saddle-surface} implies
the following.

\begin{corollary}
Let $M$ be a complete simply connected
two-dimensional Finsler manifold
which admits a saddle isometric immersion into a
vector space with a Minkowski norm. Then
every two points of $M$ 
are connected by a unique geodesic,
and all geodesics are length minimizers.
\end{corollary}

These results could make one think that Finsler
metrics of saddle surfaces have some special local properties
(such as non-positivity of some curvature-like invariants)
that imply this global properties.
However the following theorem shows
that this is not the case: every Finsler metric
(including positively curved Riemannian metrics)
can be locally realized as a metric of a saddle surface.

\begin{theorem}
\label{t-saddle-embedding}
Let $M$ be a two-dimensional Finsler manifold.
Then every point of $M$ has a neighborhood which admits
a saddle smooth isometric embedding
into a 4-dimensional normed space with a Minkowski norm.
\end{theorem}

\begin{remark}
Every $n$-dimensional Finsler manifold can be locally
isometrically embedded into a $2n$-dimensional
normed space with a Minkowski norm,
see \cite{Shen} and references therein.
Globally, every \textit{compact} Finsler manifold $M$
can be isometrically embedded in a finite dimensional
normed space $V$ but the dimension of $V$ cannot
be bounded above in terms of $\dim M$ and moreover
non-compact Finsler manifolds in general do not
admit such embeddings, see~\cite{BI93}.
\end{remark}

\begin{remark}
It is still not clear whether saddle surfaces
in \textit{3-dimensional} spaces are intrinsically
different from convex ones.
In other words, can a strictly saddle surface in a 3-dimensional
normed space (with a Minkowski norm) be locally isometric
to a strictly convex surface in another such space?

There might be obstructions to such isometries:
it seems that, unlike in the Riemannian case,
a generic Finsler metric does not admit any isometric
embeddings into 3-dimensional spaces.
So it would not be surprising that such an embedding,
if it exists at all, is essentially unique.
\end{remark}

The ``opposite'' to the class of saddle surfaces is
the class of convex surfaces. Convex surfaces
in $\R^3$ are the only surfaces such that all their affine
transformations are non-negatively curved, cf.~\cite{Shefel}.
The next theorem shows that geodesics on
complete convex surfaces in normed spaces
also possess properties typical for positive curvature.

\begin{theorem}
\label{t-convex-surface}
Let $V$ be a 3-dimensional normed space whose norm
is $C^1$-smooth and strictly convex.
Let $B\subset V$ be a convex set with nonempty interior
not containing straight lines
(in other words, $B$ is not a cylinder).
Then there are no geodesic lines in $\pd B$
(a geodesic line is a curve which is a shortest
path between any pair of its points).
\end{theorem}

The rest of the paper is organized as follows.
Theorems \ref{t-saddle-surface},
\ref{t-saddle-embedding} and \ref{t-convex-surface} are proved in sections
\ref{sec-saddle}, \ref{sec-embedding} and \ref{sec-convex},
respectively.
These sections are completely independent from one another
and each section introduces its own notation.

The proofs are mostly elementary although some parts
involve cumbersome computations.
We do not use any machinery of contemporary Finsler
geometry (beyond things like the geodesic equation
in Section \ref{sec-saddle}).
In fact, as shown by Theorem \ref{t-saddle-embedding},
this machinery would be useless here.

\section{Geodesics on saddle surfaces}
\label{sec-saddle}

The goal of this section is to prove Theorem \ref{t-saddle-surface}.

\subsection{Preliminaries and notation}
We consider a finite dimensional vector space $V$
with a Minkowski norm denoted by $\Phi$.
As usual $V^*$ denotes the dual space (that is
the space of linear functions from $V$ to $\R$).
By $\langle,\rangle$ we denote the standard pairing
between $V^*$ and $V$, that is, $\langle L,v\rangle = L(v)$
for $L\in V^*$, $v\in V$.

The dual space $V^*$ carries the dual norm $\Phi^*$
given by
$$
 \Phi^*(L) = \sup \{ \langle L,v\rangle : \Phi(v) =1 \} ,
$$
this dual norm is also smooth and quadratically
convex.
The above supremum is attained at a unique vector
from the unit sphere of $\Phi$, the direction of
this vector is referred to as the direction of maximal growth,
or the \textit{gradient direction}, of~$L$.

For a $C^1$ function $f\colon V\to\R$ and $x\in V$,
we denote by $df(x)$ the derivative of $f$ at~$x$.
This is an element of $V^*$; in our notation,
the derivative of $f$ at $x$ along a vector $v\in V$
is written as $\langle df(x),v\rangle$.
If $df(x)\ne 0$, then the \textit{gradient direction}
of $f$ at $x$ is defined as that of the co-vector $df(x)$.

\smallskip
The \textit{Legendre transform} of $\Phi$ is the map
$\L_\Phi\colon V\to V^*$ defined by
$$
 \L_\Phi(v) = \tfrac12 d\Phi^2(v) .
$$
One easily sees that this map features the following properties:

(i) it is positively homogeneous:
$\L_\Phi(tv)=t\L_\Phi(v)$ for all $v\in V$ and $t\ge 0$;

(ii) if $\Phi(v)=1$, then $\L_\Phi(v)$ is the unique linear
function $L\in V^*$ such that $\Phi^*(L)=1$ and $\langle L,v\rangle=1$;

(iii) $\L_\Phi$ preserves the norm: $\Phi^*(\L_\Phi(v))=\Phi(v)$ for all $v\in V$;

(iv) $\L_\Phi$ is a diffeomorphism between $V\setminus\{0\}$
and $V^*\setminus\{0\}$, in particular, it is a diffeomorphism
between the unit spheres of $\Phi$ and $\Phi^*$;

(v) the inverse Legendre transform $\L_\Phi^{-1}$ coincides
with the Legendre transform $\L_{\Phi^*}$
(as usual, we identify $V^{**}$ with $V$).

We use these properties without explicitly referring to them.

\smallskip

Let $\gamma\colon I\to V$, where $I\subset\R$ is an interval,
be a smooth unit-speed curve (that is, $\Phi(\dot\gamma)\equiv 1$).
The co-vector
$$
 K_\gamma(t) := \frac d{dt} \L_\Phi(\dot\gamma(t))
$$
is referred to as the \textit{curvature co-vector} of $\gamma$ at $t$.
(This co-vector takes the role of the curvature vector in the
first variation formula.) A curve $\gamma$ lying on a
smooth submanifold $M\subset V$ is a geodesic in $M$
if an only if $K_\gamma(t)$ annihilates the tangent
space $T_{\gamma(t)}M\subset V$
for all $t$ (that is, $\langle K_\gamma(t),v\rangle=0$ for
all $v\in T_{\gamma(t)}M$).

\smallskip
For a Finsler metric $\phi$ on a manifold $M$,
the notations $\phi^*$ and $\L_\phi$ denote
the fiber-wise dual norm and the fiber-wise
Legendre transform; $\phi^*$ is a function on $T^*M$
and $\L_\phi$ is a map from $TM$ to $T^*M$.
Note that, if $M\subset V$ is a smooth submanifold,
$\phi$ is the induced Finsler metric on $M$, 
$x\in M$ and $v\in T_xM$, then $\L_\phi(v)=\L_\Phi|_{T_xM}$.

\subsection{Calibrators}

Let $S\colon M\to V$ be a saddle surface
and $\gamma\colon[a,b]\to M$ a geodesic of the
induced Finsler metric $\phi$ on $M$.
We are going to prove that $\gamma$ minimizes length
among $C^0$-nearby curves. It suffices to do this assuming
that $\gamma$ is embedded (that is, has
no self-intersections in $M$).
Indeed, to reduce the general case to the special case when $\gamma$ is embedded,
construct an immersion
$$
f\colon(a-\ep,b+\ep)\times(-\ep,\ep)\to M
$$
such that $f(t,0)=\gamma(t)$ for all $t\in[a,b]$ and apply the special case
to the induced metric $f^*\phi$ on $(a-\ep,b+\ep)\times(-\ep,\ep)$ and the
geodesic $t\mapsto (t,0)$ there.

Throughout the rest of the proof we assume that $\gamma$ is embedded
and extended (as an embedded geodesic) to an interval $(a-\ep,b+\ep)$.
We abuse notation and denote the image $\gamma(a-\ep,b+\ep)\subset M$
by the same letter $\gamma$.

\begin{df}
%\label{d-calibrator}
Let $U\subset M$ be a neighborhood of $\gamma([a,b])$.
A map $h\colon U\to\R$ is said to be a \textit{calibrator}
for $\gamma$ if the following holds:

(i) $h(\gamma(t))=t$ for all $t\in(a-\ep,b+\ep)$ such that $\gamma(t)\in U$;

(ii) $\phi^*(dh(x))\le 1$ for all $x\in U$.
\end{df}

If there is a calibrator for $\gamma$ defined on a neighborhood $U$,
then $\gamma|_{[a,b]}$ is a unique shortest path
in $U$ between $\gamma(a)$ and $\gamma(b)$.
Indeed, let $\gamma_1\colon[c,d]\to U$ be a piecewise smooth path
with the same endpoints. Then
$$
\begin{aligned}
 \len(\gamma_1)
 &=\int_c^d \phi(\dot\gamma_1(t))\,dt
 \ge \int_c^d\langle dh(\gamma_1(t)),\dot\gamma_1(t)\rangle\,dt
 = \int_c^d \frac d{dt} h(\gamma_1(t))\,dt \\
 &= h(\gamma_1(d))-h(\gamma_1(c))= h(\gamma(b))-h(\gamma(a))= b-a
 = \len(\gamma|_{[a,b]}) .
\end{aligned}
$$
Here we used the fact that $\phi^*(dh(x))\le 1$ for all $x\in U$
and hence $\langle dh(x),v\rangle \le \phi(v)$ for all $v\in T_xM$.

\begin{df}
\label{d-almost-calibrator}
Let $U\subset M$ be a neighborhood of $\gamma([a,b])$.
A map $h\colon U\to\R$ is said to be an \textit{almost calibrator}
for $\gamma$ if the following holds:

(i) $h(\gamma(t))=t$ for all $t\in(a-\ep,b+\ep)$ such that $\gamma(t)\in U$;

(ii) $\phi^*(dh(x))\le 1 + o(\dist(x,\gamma)^2)$ as $\dist(x,\gamma)\to 0$.
\end{df}

\begin{lemma}
\label{l-almost-calibrator}
If $\gamma$ admits an almost calibrator, then $\gamma|_{[a,b]}$
is a shortest path in some neighborhood of its image.
\end{lemma}

\begin{proof}
By the definition of almost calibrator,
we have $\langle dh(\gamma(t)),\dot\gamma(t)\rangle=1$
and $\phi^*(dh(\gamma(t))\le 1$ for all $t$.
Hence $\phi^*(dh(\gamma(t))=1$
and $dh(\gamma(t))=\L_\phi(\dot\gamma(t))$.
We may assume that $dh\ne 0$ on $U$.

Define a vector field $V$ on $U$ by
$$
V(x)=\L_\phi^{-1}(dh(x))=\L_{\phi^*}(dh(x)), \qquad x\in U .
$$
For any co-vector $\xi\in T^*M$ such that $\langle \xi,V(x)\rangle$,
the derivative of $\phi^*_x$ at $dh(x)\in T^*M$ along $\xi$
equals zero (this follows from the definition of the Legendre
transform $\L_{\phi^*}$). Therefore
\begin{equation}
\label{e-xi}
 \phi^*(dh(x)+\xi) \le \phi^*(dh(x)) + C\|\xi\|^2
\end{equation}
for some constant $C$, all $x\in U$ sufficiently close to $\gamma([a,b])$,
and all $\xi\in T^*M$ such that $\langle\xi,V(x)\rangle=0$.

Recall that $V(\gamma(t))=\dot\gamma(t)$ for all $t$,
so $\gamma$ is a trajectory of $V$.
Hence if $U$ is sufficiently small, there is a smooth map
$f\colon U\to\R$ such that $df\ne 0$ and $f$ is constant along the
trajectories of $V$ or, equivalently,
$\langle df(x), V(x)\rangle=0$ for all $x\in U$.
We may assume that $f=0$ on $\gamma$,
then
$$
  c\cdot\dist(x,\gamma)\le f(x)\le C\cdot\dist(x,\gamma)
$$
for some constants $c,C>0$ and all $x\in U$.
Define a function $g\colon U\to\R$ by
$$
 g(x) = \left(1-\sigma f(x)^2\right)\cdot h(x)
$$
for a small $\sigma>0$. Note that $g=f$ on $\gamma$. We have
$$
 dg(x) = \left(1-\sigma f(x)^2\right) 
 \left( dh(x) - \frac{2\sigma f(x)}{1-\sigma f(x)^2}\cdot df(x) \right) .
$$
Since $\langle df(x), V(x)\rangle=0$, we can apply
\eqref{e-xi} to
$$
 \xi = - \frac{2\sigma f(x)}{1-\sigma f(x)^2}\cdot df(x) .
$$
This yields
\begin{equation}
\label{e-xi1}
 \phi^*(dg(x)) \le \left(1-\sigma f(x)^2\right)\cdot\phi^*(dh(x))
 + C\cdot\frac{4\sigma^2 f(x)^2}{1-\sigma f(x)^2}\cdot \|df(x)\|^2.
\end{equation}
We may assume that $U$ is so small that $\sigma f(x)^2<1/2$
for all $x\in U$ and $\|df\|$ is bounded on~$U$.
Then the second summand in \eqref{e-xi1}
is bounded above by $C_1\sigma^2 f(x)^2$ for some constant $C_1>0$.
By the assumption (ii) of Definition \ref{d-almost-calibrator},
we have
$$
 \phi^*(dh(x)) \le 1 + o(\dist(x,\gamma)^2) = 1 + o(f(x)^2), \qquad \dist(x,\gamma)\to 0.
$$
Hence we have the following estimate for the first summand in \eqref{e-xi1}:
$$
 (1-\sigma f(x)^2)\cdot\phi^*(dh(x)) \le 1-\tfrac12\sigma f(x)^2
$$
for all $x$ sufficiently close to $\gamma$.
Thus \eqref{e-xi1} implies that
$$
 \phi^*(dg(x)) \le 1-\tfrac12\sigma f(x)^2 + C_1\sigma^2 f(x)^2
 = 1 - \tfrac12\sigma f(x)^2 (1-2C_1\sigma)
$$
for all $x$ from a neighborhood $U'\subset U$ of $\gamma([a,b])$.
Hence $\phi^*(dg(x))\le 1$ for all $x\in U'$
provided that $\sigma<(2C_1)^{-1}$.
Thus $g$ is a calibrator for $\gamma$ in $U'$,
therefore $\gamma|_{[a,b]}$ is a shortest path in $U'$.
\end{proof}

\subsection{The construction}
Our goal is to construct an almost calibrator $h$ for an embedded
geodesic $\gamma$ on our saddle surface.
Recall that our surface is parameterized by $S\colon M\to V$.
Let $\gamma_S=S\circ\gamma$.
We define $h\colon U\to\R$, where $U$ is a neighborhood of $\gamma([a,b])$,
by the following implicit relation:
the value $h(x)$ is a parameter $t\in (a-\ep,b+\ep)$ such that
\begin{equation}
\label{e-hdef}
\langle \L_\Phi(\dot\gamma_S(t)),S(x)-\gamma_S(t)\rangle = 0 .
\end{equation}
Observe that for $x=\gamma(t)$ this equation is satisfied
and the derivative of its left-hand side with respect to $t$
is nonzero (more precisely, it equals $-1$).
Hence by the Implicit Function Theorem there exists
a neighborhood $U$ of $\gamma$ and a unique smooth function
$h\colon U\to\R$ such that $h(\gamma(t))=t$ for all $t$
and \eqref{e-hdef} holds for every $x\in U$ and $t=h(x)$.

We are going to show that $h$ is an almost calibrator for $\gamma$.
The first requirement of Definition \ref{d-almost-calibrator}
is immediate from the construction.
The second requirement is local; it suffices to verify it
in a small neighborhood of every point of $\gamma$.
Therefore we may assume that our surface is embedded and
identify $M$ with its image in the space.
That is, $M=U$ is a submanifold of $V$
and $S$ is the inclusion map $M\to V$. Then \eqref{e-hdef}
takes the form
\begin{equation}
\label{e-hdef-simple}
 \langle \L_\Phi(\dot\gamma(t)),x-\gamma(t)\rangle = 0
\end{equation}
where $x\in M\subset V$, $t=h(x)$.

\subsubsection*{Riemannian case}
Before proving that $h$ is an almost calibrator, we briefly
explain why this is true in the case when the ambient space is Euclidean.
First observe that the condition (ii) in the definition
of almost calibrator depends only on the derivatives of $h$
at $\gamma$ up to the second order.
By \eqref{e-hdef-simple}, every level set $h^{-1}(t)$ of $h$
is the intersection of $M$ with the hyperplane orthogonal
to $\gamma$ at $\gamma(t)$. This normal section of the surface
has zero geodesic curvature at $\gamma(t)$, therefore it
suffices to prove the result for a similar function whose
level sets are geodesics orthogonal to~$\gamma$.
Since the Gaussian curvature of the surface in nonpositive,
these geodesics diverge from one another, hence the distance
between level sets is minimal at the base curve $\gamma$.
This implies that the norm of the derivative of our function
attains its minimum (equal to~1) at $\gamma$, hence the result.

\subsection{Computations}

\begin{lemma}
\label{l-diff0}
Let $x_0=\gamma(t_0)$ where $t_0\in(a-\ep,b+\ep)$.
Then
$$
dh(x_0)=\L_\Phi(\dot\gamma(t_0))|_{T_{x_0}M}=\L_\phi(\dot\gamma(t_0))
$$
and therefore $\phi^*(dh(x_0))=1$.
\end{lemma}

\begin{proof}
Recall that $h(x_0)=t_0$. By \eqref{e-hdef-simple} we have
$$
 \langle \L_\Phi(\dot\gamma(h(x))),x-\gamma(h(x))\rangle = 0
$$
for all $x\in M$.
Differentiate this identity at $x=x_0$ along a vector
$v\in T_{x_0}M$.
Since the second term $x-\gamma(h(x))$ of the above product
is zero for $x=x_0$, the derivative of the first term
cancels out, and the differentiation yields
$$
 \langle \L_\Phi(\dot\gamma(t_0)),v-\dot\gamma(t_0) h'_v\rangle = 0
$$
where $h'_v$ is the derivative of $h$ at $x_0$ along $v$,
that is $h'_v = \langle dh(x_0), v\rangle$.
Since $\langle \L_\Phi(\dot\gamma(t_0)),\dot\gamma(t_0)\rangle=1$,
it follows that
$h'_v = \langle \L_\Phi(\dot\gamma(t_0)),v\rangle$.
Since $v$ is an arbitrary vector from $T_{x_0}M$, it follows that
$$
dh(x_0)=\L_\Phi(\dot\gamma(t_0))|_{T_{x_0}M} .
$$
Since $\dot\gamma(t_0)$ is tangent to the surface, the right-hand side
equals $\L_\phi(\dot\gamma(t_0))$.
The identity $\phi^*(dh(x_0))=1$ now follows from the fact
that $\phi^*(\L_\phi(v))=\phi(v)$ for every $v\in TM$.
\end{proof}

Now we introduce a special coordinate system $(t,s)$
in a neighborhood of $\gamma$. The $s$-coordinate lines
of this system are level curves of~$h$. The $t$-coordinate lines
are ``gradient curves'' of $h$ (that is, curves tangent to
the vector field $\L_\phi^{-1}(dh)$), in particular,
$\gamma$ itself is the $t$-coordinate line corresponding to $s=0$.

More precisely, let $r\colon(a-\ep,b+\ep)\times(-\ep,\ep)\to M\subset V$
be a local parameterization (whose argument is 
denoted by $(t,s)$) such that for all $(t,s)$ the following holds:

(1) $r(t,0)=\gamma(t)$;

(2) $h(r(t,s))=t$;

(3) the first partial derivative $r'_t$ of $r$ at $(t,s)$ is proportional
to the vector $\L_\phi^{-1}(dh(x))$ where $x=r(t,s)$.

Lemma \ref{l-diff0} ensures that these conditions are compatible.
The third condition means that the vector $r'_t$ points in the direction
of the maximal growth of $h$. Since the derivative of $h$ along this vector
equals~1 (by the second condition), it follows that
$$
 \phi^*(dh(x)) = \frac 1{\phi(r'_t(t,s))} = \frac 1{\Phi(r'_t(t,s))}
$$
for $x=r(t,s)$. Therefore the requirement (ii)
of Definition \ref{d-almost-calibrator} for $h$
is equivalent to the following:
$$
 \Phi(r'_t(t,s)) \ge 1 - o(s^2), \qquad s\to 0 .
$$
Denote
$$
\rho(t,s) = \Phi^2(r'_t(t,s)).
$$
Now it suffices to prove that
$$
\rho(t,s) \ge 1 - o(s^2), \qquad s\to 0 .
$$
By Lemma \ref{l-diff0} we have $\rho(t,0)=1$ for all $t$,
therefore it suffices to prove that
$\rho'_s(t,0)=0$ and $\rho''_{ss}(t,0)\ge 0$
for all $t$.

Fix $t_0\in(a-\ep,b-\ep)$ and let us verify that
$\rho'_s(t_0,0)=0$ and $\rho''_{ss}(t_0,0)\ge 0$.
We introduce the following notation:
$$
\begin{aligned}
 & x_0 = r(t_0,0) = \gamma(t_0), \\
 & v(t,s) = r'_t(t,s), \\
 & v_0 = v(t_0,0) = \dot\gamma(t_0), \\
 & L = \L_\Phi(v_0) = \L_\Phi(\dot\gamma(t_0)), \\
 & K = \tfrac d{dt}\big|_{t=t_0} \L_\Phi(\dot\gamma(t))
\end{aligned}
$$
Recall that $K\in V^*$ is the ``curvature co-vector''
of $\gamma$ at $t_0$ and it annihilates $T_xM$
(since $\gamma$ is a geodesic).
Using this notation, the definition of $\rho$ can be written as
$$
 \rho(t,s) = \Phi^2(v(t,s)) .
$$

\begin{lemma}
\label{l-Lv}
For all $s\in(-\ep,\ep)$ we have
\begin{align}
\label{e-Lv1}
&\langle L, v'_s(t_0,s) \rangle =  - \langle K, r'_s(t_0,s)\rangle , \\
\label{e-Lv1-0}
&\langle L, v'_s(t_0,0) \rangle =  0 , \\
\label{e-Lv2}
&\langle L, v''_{ss}(t_0,0) \rangle = - \langle K, r''_{ss}(t_0,0)\rangle .
\end{align}
\end{lemma}

\begin{proof}
The fact that $h(r(t,s))=t$ and \eqref{e-hdef-simple} imply that
$$
 \langle \L_\Phi(\dot\gamma(t)), r(t,s)-r(t,0) \rangle = 0
$$
for all $t,s$.
Differentiating this with respect to $t$ yields
$$
 \big\langle \tfrac d{dt}\L_\Phi(\dot\gamma(t)), r(t,s)-r(t,0) \big\rangle
 +\big\langle \L_\Phi(\dot\gamma(t)), r'_t(t,s)-r'_t(t,0) \big\rangle = 0 .
$$
Since
$$
 \langle \L_\Phi(\dot\gamma(t)), r'_t(t,0) \rangle
 = \langle \L_\Phi(\dot\gamma(t)), \dot\gamma(t) \rangle = 1,
$$
it follows that
$$
\big\langle \tfrac d{dt}\L_\Phi(\dot\gamma(t)), r(t,s)-r(t,0) \big\rangle
 +\big\langle \L_\Phi(\dot\gamma(t)), r'_t(t,s) \big\rangle - 1 = 0,
$$
or, equivalently,
$$
 \big\langle \L_\Phi(\dot\gamma(t)), v(t,s) \big\rangle
 = 1 - \big\langle \tfrac d{dt}\L_\Phi(\dot\gamma(t)), r(t,s)-r(t,0) \big\rangle .
$$
Substituting $t=t_0$ and using the definitions of $L$ and $K$
yields
$$
\langle L, v(t_0,s) \rangle = 1 - \langle K, r(t_0,s)-r(t_0,0)\rangle .
$$
Differentiating this with respect to $s$ yields \eqref{e-Lv1}.
Since $r'_s(t_0,0)$ is a tangent vector to $M$ at $x_0$,
we have $\langle K, r'_s(t_0,0)\rangle=0$, hence
substituting $s=0$ into \eqref{e-Lv1} yields \eqref{e-Lv1-0}.
Finally, differentiating \eqref{e-Lv1} with respect to $s$
at $s=0$ yields \eqref{e-Lv2}.
\end{proof}

Recall that
$$
 L = \L_\Phi(v_0) = \tfrac12 d\Phi^2(v_0)
$$
by the definitions of $L$
and Legendre transform.
Now we can verify that ${\rho'_s(t_0,0)=0}$:
$$
\rho'_s(t_0,0) = \tfrac d{ds}\big|_{s=0} \Phi^2(v(t_0,s)) 
= \langle d\Phi^2(v_0), v'_s(t_0,0)\rangle
= 2 \langle L, v'_s(t_0,0)\rangle = 0
$$
(the last identity follows from \eqref{e-Lv1-0}).

Define a quadratic form $Q$ on $V$ by 
$$
 Q=\tfrac12 d^2\Phi^2(v_0)
$$
(this is the second derivative at $v_0$
of the function $v\mapsto \frac12\Phi^2(v)$ on $V$).
Since $\Phi$ is a quadratically convex norm,
$Q$ is positive definite.
We use $Q$ as an auxiliary Euclidean structure on $V$.

From the definitions, for any $w\in V$ we have
$$
 \langle K, w\rangle 
 = \big\langle \tfrac d{dt}\big|_{t=t_0}\L_\Phi(\dot\gamma(t)), w \big\rangle
 = Q(\ddot\gamma(t_0),w)
$$
since $\L_\Phi(\dot\gamma(t))=\tfrac12 d\Phi^2(\dot\gamma(t))$ and
$\dot\gamma(t_0)=v_0$.
In particular, the vector $\ddot\gamma(t_0)$ is $Q$-orthogonal
to the tangent plane $T_{x_0}M$.
Let $n$ be a $Q$-unit vector which is $Q$-orthogonal to $T_{x_0}M$
and proportional to $\ddot\gamma(t_0)$ if the latter is nonzero.
Then
\begin{equation}
\label{e-KQ}
\langle K, w\rangle = Q(\ddot\gamma(t_0),w)
= Q(\ddot\gamma(t_0),n)\cdot Q(w,n)
\end{equation}
for every $w\in V$.
Now we compute $\rho''_{ss}(t_0,0)$ as follows:
\begin{equation}
\label{e-rhoss}
 \tfrac12\rho''_{ss}(t_0,0) = \frac{d^2}{ds^2}\bigg|_{s=0} \tfrac12 \Phi^2(v(t_0,s))
 = Q(v'_s,v'_s) + L(v''_{ss})
\end{equation}
where the partial derivatives $v'_s$ and $v''_{ss}$
are taken at $(t_0,0)$.
By \eqref{e-Lv2}, at $(t,s)=(t_0,0)$ we have
$$
 L(v''_{ss}) = - \langle K, r''_{ss}\rangle 
 = - Q(\ddot\gamma(t_0),n)\cdot Q(r''_{ss},n)
 = - Q(r''_{tt},n)\cdot Q(r''_{ss},n)
$$
where the second identity follows from \eqref{e-KQ}.
Using this identity and the fact that $v'_s=r''_{ts}$,
we rewrite \eqref{e-rhoss} as follows:
$$
 \tfrac12\rho''_{ss}(t_0,0) = Q(r''_{ts},r''_{ts})- Q(r''_{tt},n)\cdot Q(r''_{ss},n) .
$$
With the trivial estimate $Q(r''_{ts},r''_{ts})\ge Q(r''_{ts},n)^2$,
this implies
$$
 \tfrac12\rho''_{ss}(t_0,0) \ge Q(r''_{ts},n)^2 - Q(r''_{tt},n)\cdot Q(r''_{ss},n) .
$$
The right-hand side is minus the determinant of the second fundamental form
of~$M$ with respect to the Euclidean structure $Q$ and the normal vector $n$.
Since $M$ is a saddle surface, this determinant is nonpositive and we conclude that
$$
 \rho''_{ss}(t_0,0) \ge 0 .
$$
As explained above, this inequality implies that $h$
is an almost calibrator for $\gamma$
and therefore
(by Lemma \ref{l-almost-calibrator})
 $\gamma$ is a shortest path in a neighborhood of
$\gamma([a,b])$.
This finishes the proof of Theorem \ref{t-saddle-surface}.

\section{Existence of saddle embeddings}
\label{sec-embedding}

The goal of this section is to prove Theorem \ref{t-saddle-embedding}.
Our plan is the following. First we define a saddle map
$F\colon U\to\R^4$, where $U$ is a small neighborhood
of a point, and then we define a norm on $\R^4$
such that $F$ is an isometric embedding with
respect to this norm.
For such a norm to exist, the images of $\phi$-unit vectors
under $dF$ should lie on a smooth strictly convex
hypersurface in $\R^4$ (this surface can be taken
for the unit sphere of the norm that we want to construct).
Our construction ensures that $dF$ restricted to the
set of $\phi$-unit vectors parameterizes a
strictly convex hypersurface located in a small neighborhood
of a plane. Then a separate construction
(described in the first subsection) is used
to extend this surface to a compact smooth strictly convex
hypersurface that can be taken for the unit sphere of a norm.

\subsection{Extending a convex surface}

\begin{df}
\label{d-pre-convex}
Let $\Sigma\subset\R^n$ be a smooth embedded hypersurface.
We say that $\Sigma$ is \textit{pre-convex} if
for every $p\in \Sigma$ there is a linear function $L\colon\R^n\to\R$
such that
\begin{equation}
\label{e-pre-convex}
 L(q) \le L(p) - c\cdot|p-q|^2
\end{equation}
for some constant $c>0$ and all $q\in\Sigma$.
\end{df}

\begin{remark}
The function $L$ satisfying \eqref{e-pre-convex}
is unique up to multiplication by a constant:
it must be zero on the tangent space $T_p\Sigma\subset\R^n$.
\end{remark}

\begin{remark}
\label{r-pre-convex}
If \eqref{e-pre-convex} holds for all $q$ close to $p$,
then the second fundamental form of $\Sigma$ at $p$ (with respect to
a suitable normal vector) is positive definite.
Conversely, if the second fundamental form of $\Sigma$ at $p$
is positive definite, then \eqref{e-pre-convex}
holds for all $q$ from a sufficiently small neighborhood of~$p$.

It follows that, if the requirement of Definition \ref{d-pre-convex}
is satisfied for all $p$ from a compact set $K\subset\Sigma$,
then some neighborhood of $K$ in $\Sigma$ is pre-convex.
\end{remark}

\begin{lemma}
\label{l-convex-extension}
Let $\Sigma\subset\R^n$ be a pre-convex hypersurface and $K\subset \Sigma$
a compact set. Then there exists a compact convex surface $\Sigma'$
(that is, a boundary of a convex body) which is smooth,
quadratically convex, and contains a neighborhood of $K$ in~$\Sigma$.

Furthermore, if $\Sigma$ is symmetric with respect to the origin,
then $\Sigma'$ can be chosen symmetric as well.
\end{lemma}

\begin{proof}
This is a standard type of argument, so we limit ourselves to a sketch.
First of all, there is a neighborhood of $K$ whose closure $K_1$
is compact and contained in~$\Sigma$.
The most natural thing would be to take the convex hull of $K_1$, and the
only problems would be that it is not necessarily smooth and quadratically convex.

It is easy to make it quadratically convex by taking the intersection $B_1$ of all balls 
of radius $R$ containing $K_1$, where $R$ is larger that the reciprocal of normal
curvatures of $\Sigma$ over $K_1$. Then, by choosing $\ep>0$ smaller than
the reciprocal of normal curvatures of $\Sigma$ over $K_1$ and taking the inward 
$\ep$-equidistant of the surface of $B_1$ and then the outward $\ep$-equidistant
of the result, we obtain a surface of a body $B_2$ which contains $K_2$, 
quadratically convex and $C^1$-smooth; furthermore, its principal
curvatures are bounded between $1/R$ and $1/\ep$ in the barrier sense.  

All is left is to smoothen this surface further.
This is done in a standard way by covering the surface
by two open sets one of which contains $K$
and the other does not intersect~$K$.
Then one approximates the radial function of $B_2$
on the second set using convolutions.
Sufficiently close approximations (with derivatives) will preserve
quadratic convexity, and one concludes the argument by gluing these
approximations with the original surface in a neighborhood of $K$
using a partition of unity.
\end{proof}

\subsection{The case of constant metric}

For a Finsler metric $\phi$ in a region $U\subset\R^2$
we denote by $S_\phi U$ the set of all $\phi$-unit vectors
in $TU=U\times\R^2$, that is,
$$
S_\phi U = \{v\in TU:\phi(v)=1\} .
$$
Clearly $S_\phi U$ is a smooth 3-dimensional submanifold
of $TU$.

We say that a Finsler metric on $U\subset \R^2$ is \textit{constant}
if it does not depend on a point, that is there is a norm
$\|\cdot\|$ on $\R^2$ such that $\phi(x,v)=\|v\|$ for all
$x\in U$, $v\in T_xU\simeq\R^2$. Of course this is a coordinate-dependent 
definition (though invariant under affine coordinate changes), however
every flat Finsler metric locally admits a coordinate system in which
it is constant.

\begin{lemma}
\label{l-constant-metric}
For every constant Finsler metric $\phi$ on $\R^2$ there
exist a neighborhood $U\subset\R^2$ of the origin and a
smooth saddle embedding $F\colon U\to\R^4$ such that
the map $dF|_{S_\phi U}$ is an embedding and its image
$dF(S_\phi U)$ is a pre-convex surface in $\R^4$.
\end{lemma}

\begin{proof}
Let $B$ be the unit ball of $\phi$ centered at 0 and $S=\pd B$.
Then $S_\phi U=U\times S$ for any open set $U\subset\R^2$.

There is a parallelogram $P$ containing $B$ such that
the midpoints of its for sides are on~$S$
(for example, consider a minimum area parallelogram
containing~$B$). Introduce a new coordinate system
$(x,y)$ in the plane such that in these coordinates
$$
 P = \{(x,y): x,y\in[-1,1] \}.
$$
Now $B\subset P=[-1,1]^2$ and $B$ contains the four
points $(\pm1,0)$ and $(0,\pm1)$.

For every $\sigma > 0$, define a map
$$
F_\sigma\colon \R^2 \to \R^4 = \R^2\times\R\times\R
$$
by
$$
 F_\sigma(x,y) = (f_\sigma(x,y), x^2-y^2, xy)
$$
where
$$
 f_\sigma(x,y) = (1-\sigma^2x^2-\sigma^2y^2)\cdot(x-\sigma x^3,y-\sigma y^3) \in\R^2 .
$$
Notice that $F_\sigma$ converge to $F_0$ as $\sigma \rightarrow 0$, where 
$$F_0(x,y)=(x,y, x^2-y^2, xy).$$
Observe that the derivative of $F_\sigma$ at the origin is the
standard inclusion of $\R^2$ into $\R^4$: $(\xi,\eta)\mapsto(\xi,\eta,0,0)$.
Therefore $F_\sigma$ when restricted to a small neighborhood of the origin
is a smooth embedding.
We are going to show that $F_\sigma$ satisfies the requirements
for $F$ for all sufficiently small $\sigma>0$.

First we prove that, if $\sigma$ is sufficiently small and
$U\subset\R^2$ is a sufficiently small neighborhood of the origin,
then ${F_\sigma}
{|_U}$ is strictly saddle and $dF_\sigma|_{S_\phi U}$
is an embedding. 
Since $F_\sigma$ converges to $F_0$ with the derivatives as $\sigma\to 0$,
it suffices to verify these facts for $F_0$.

Let us show that $F_0$ is strictly saddle at the origin
(by continuity of the second fundamental form, this
implies that it is strictly saddle near the origin).
For a unit vector  $\nu$ normal to $F_0$ at the origin,
denote by $Q_\nu$ the second fundamental form of $F_0$
with respect to~$\nu$. A unit normal vector $\nu$
can be written as $\nu=\alpha e_3+\alpha e_4$
where $\alpha^2+\beta^2=1$.
Then $Q_\nu=\alpha Q_{e_3}+\beta Q_{e_4}$,
and the quadratic forms $Q_{e_3}$ and $Q_{e_4}$ are given by
$$
 Q_{e_3}(x,y) = x^2-y^2, \qquad Q_{e_3}(x,y) = xy
$$
for all $x,y\in\R$. The forms $Q_{e_3}$ and $Q_{e_4}$ are
linearly independent, hence $Q_\nu\ne 0$.
Furthermore, since $Q_{e_3}$ and $Q_{e_4}$ are traceless, so is $Q_\nu$,
and thus $Q_\nu$ is indefinite. Hence $F_0$ is saddle at~0.

Now we show that $dF_0|_{S_\phi U}$ is an embedding
provided that $U\subset\R^2$ is a sufficiently small neighborhood of~0.
For brevity, we denote $dF_0\colon T\R^2=\R^2\times\R^2\to\R^4$
by $G$. In coordinates, $G$ is given by
$$
 G(x,y,\xi,\eta) = (\xi,\eta,2(x\xi-y\eta),2(x\eta+y\xi))
$$
where $(x,y)$ are coordinates in $\R^2$ and
$(\xi,\eta)$ are coordinates in $T_{(x,y)}\R^2$.

Recall that $S_\phi U=U\times S$ and observe that
$dF_0|_{\{0\}\times S}$ is injective.
Therefore it suffices to verify that
the partial derivatives of $G$ at every point
of $\{0\}\times S$ are linearly independent.
And this is trivial because
$$
\begin{aligned}
 \frac{\pd G}{\pd x}(x,y,\xi,\eta) &= (0,0,2\xi,2\eta), \\
 \frac{\pd G}{\pd y}(x,y,\xi,\eta) &= (0,0,-2\eta,2\xi), \\
 \frac{\pd G}{\pd \xi}(x,y,\xi,\eta) &= (1,0,2x,2y), \\
 \frac{\pd G}{\pd \eta}(x,y,\xi,\eta) &= (0,1,-2y,2x), \\
\end{aligned}
$$
so $\det(dG)=\xi^2+\eta^2$,
and $(\xi,\eta)\ne(0,0)$ if $(\xi,\eta)\in S$.

\medskip
It remains to show that the set
$\Sigma:=dF_\sigma({S_\phi U})=dF_\sigma(U\times S)$
is pre-convex for some $\sigma>0$ and some
neighborhood $U\subset\R^2$ of the origin.
We are going to show that for every sufficiently small $\sigma$
there exists $U$ such that $\Sigma$ is pre-convex.
In other words, we assume that $\sigma\ll 1$
and $|x|,|y|\ll\sigma$ for all $(x,y)\in U$.  
By Remark \ref{r-pre-convex}, it suffices to verify
that the requirement of Definition \ref{d-pre-convex}
is satisfied for every point $p\in dF_\sigma(\{0\}\times S)$.

Let 
$p=dF_\sigma(0,0,\xi_0,\eta_0)=(\xi_0,\eta_0,0,0)$
where $v_0:=(\xi_0,\eta_0)\in S$.
Let $L_0\colon\R^2\to\R$ be the supporting linear
function for $B$ at $v_0$,
that is, 
$L_0(v)\le 1$ for all $v\in B$ and
$L_0(v_0)=1$.
Since $\phi$ is a quadratically convex norm,
we have
\begin{equation}
\label{e-L0-on-plane}
 L_0(v) \le 1 - c_0\cdot|v-v_0|^2
\end{equation}
for some $c_0>0$ and all $v\in B$.
Define $L\colon\R^4\to\R$ by
$$
 L(x,y,z,t) = L_0(x,y) .
$$
We are going to show that $L$ satisfies \eqref{e-pre-convex}
for all $q\in dF_\sigma(U\times S)$.
Since we have already verified that $dF_\sigma|_{U\times S}$
is a smooth embedding, it suffices to show that
$$
 L(dF_\sigma(x,y,\xi,\eta)) \le 1 - c\cdot (x^2+y^2+(\xi-\xi_0)^2+(\eta-\eta_0)^2)
$$
for some $c>0$ and all $(x,y)\in U$, $(\xi,\eta)\in S$.
Note that
$$
L(dF_\sigma(x,y,\xi,\eta)) = L_0(df_\sigma(x,y,\xi,\eta))
$$
by the definitions of $L$ and $F_\sigma$, so we need to show that
\begin{equation}
\label{e-L0-goal}
 L_0(df_\sigma(x,y,\xi,\eta))\le 1 - c\cdot (x^2+y^2+(\xi-\xi_0)^2+(\eta-\eta_0)^2)
\end{equation}
for some $c>0$.

Since the definition of $df_\sigma$ is symmetric with respect to
the changes $x\mapsto-x$, $y\mapsto-y$ and $x\leftrightarrow y$,
it suffices to consider the case when $\xi_0\ge\eta_0\ge 0$.
Since $B$ is inscribed in the square $[-1,1]^2$ and touches
its sides at the points $(1,0)$ and $(0,1)$, the assumption
$\xi_0\ge\eta_0\ge 0$ implies that $\xi_0\ge\frac12$
and the function $L_0$ has the form $L_0(x,y)=ax+by$
where
$$
 a = L_0(1,0) \in(0,1]
$$
and
$$
 b = L_0(0,1) \in [0,1).
$$
Moreover the coefficient $a$ is
bounded from below by a constant $a_0>0$ determined by the shape of $B$,
since the only supporting functions vanishing at $(1,0)$
are those at the points $(0,\pm 1)\in S$, and these
points are separated away from the range $\{\xi_0\ge\eta_0\ge 0\}$
that we restrict ourselves to.

Differentiating the definition of $f_\sigma$:
$$
 f_\sigma(x,y) = (1-\sigma^2x^2-\sigma^2y^2)\cdot(x-\sigma x^3,y-\sigma y^3)
$$
yields
$$
\begin{aligned}
 \frac{\pd f_\sigma}{\pd x}(x,y)
 &= (1-\sigma^2x^2-\sigma^2y^2)\cdot (1-3\sigma x^2,0) 
 - 2\sigma^2 x(x-\sigma x^3,y-\sigma y^3) \\
 &= (1-A_{11},-A_{21})
\end{aligned}
$$
where
$$
\begin{aligned}
 A_{11} &= 3\sigma x^2 + \sigma^2x^2(3-5\sigma x^2) + \sigma^2y^2(1-3\sigma x^2) , \\
 A_{21} &= 2\sigma^2xy(1-\sigma y^2)
\end{aligned}
$$
and, similarly,
$$
 \frac{\pd f_\sigma}{\pd y}(x,y) = ( -A_{12}, 1-A_{22})
$$
where
$$
\begin{aligned}
 A_{12} &= 2\sigma^2xy(1-\sigma x^2), \\
 A_{22} &=  3\sigma y^2 + \sigma^2y^2(3-5\sigma y^2) + \sigma^2x^2(1-3\sigma y^2) .
\end{aligned}
$$
Now for every $(\xi,\eta)\in S$ we have
$$
 df_\sigma(x,y,\xi,\eta) = (\xi,\eta) - (A_1,A_2)
$$
where
$$
 A_1 = \xi A_{11}+\eta A_{12},\qquad A_2= \xi A_{21}+\eta A_{22}
$$
and hence
\begin{equation}
\label{e-L0-partial}
\begin{aligned}
 L_0(df_\sigma(x,y,\xi,\eta)) 
 &= L_0(\xi,\eta) - L_0(A_1,A_2) \\
 &\le 1 - c_0(\xi-\xi_0)^2 - c_0(\eta-\eta_0)^2  - L_0(A_1,A_2)
\end{aligned}
\end{equation}
by \eqref{e-L0-on-plane}. If $(\xi,\eta)$ is separated away
from $(\xi_0,\eta_0)$ by a constant (e.g.\ by $\frac1{10}$),
this inequality implies \eqref{e-L0-goal},
since $A_{ij}$ are small when $\sigma$, $|x|$ and $|y|$ are small.
Thus we may assume that $(\xi,\eta)$ is $\frac1{10}$-close
to $(\xi_0,\eta_0)$ and therefore $\xi\ge\frac13$.
We need to estimate from below the term $L_0(A_1,A_2)$
in  \eqref{e-L0-partial}. Recall that
\begin{equation}
\label{e-L0-Aij}
 L_0(A_1,A_2) = aA_1+bA_2 = a\xi A_{11} +a\eta A_{12} + b\xi A_{21} +b\eta A_{22} .
\end{equation}
Assuming $\sigma,|x|,|y|<\frac1{10}$, we estimate
\begin{equation}
\label{e-A12}
\begin{aligned}
 |a\eta A_{12}| &\le |A_{12}| \le \sigma^2xy , \\
 |b\xi A_{21}| &\le |A_{21}| \le \sigma^2xy
\end{aligned}
\end{equation}
(since $|a|,|b|,|\xi|,|\eta|\le 1$), and
$$
 A_{11} \ge 3\sigma x^2 +\tfrac23\sigma^2y^2 
 \ge \sigma x^2 +\tfrac16\sigma^2y^2 + 2\sigma^{3/2} xy
$$
where the last inequality follows from the 
Cauchy inequality applied to $2\sigma x^2$
and $\tfrac12\sigma^2y^2$, namely
$2\sigma x^2 +\tfrac12\sigma^2y^2\ge 2\sigma^{3/2} xy$.
Since $a\ge a_0$, $\xi\ge\tfrac13$, and
$$
  2\sigma^{3/2} xy = \sigma^{-1/2} (2\sigma^2xy)
  \ge \sigma^{-1/2} |a\eta A_{12} + b\xi A_{21}|
$$
by \eqref{e-A12}, it follows that
$$
 a\xi A_{11} \ge c_1\sigma^2(x^2+y^2) + c_2\sigma^{-1/2} |a\eta A_{12} + b\xi A_{21}|
$$
where $c_1=a_0/18$ and $c_2=a_0/3$. Assuming $\sigma<c_2^{-2}$, it follows that
$$
 a\xi A_{11} +a\eta A_{12} + b\xi A_{21} \ge c_1\sigma^2(x^2+y^2) .
$$
This and \eqref{e-L0-Aij} imply that
$$
L_0(A_1,A_2) \ge c_1\sigma^2(x^2+y^2)+b\eta A_{22}
$$
and then from \eqref{e-L0-partial} we have
\begin{equation}
\label{e-L0-almost}
 L_0(df_\sigma(x,y,\xi,\eta)) \le
 1 - c_0(\xi-\xi_0)^2 - c_0(\eta-\eta_0)^2  - c_1\sigma^2(x^2+y^2) - b\eta A_{22} .
\end{equation}
To achieve our goal \eqref{e-L0-goal}, it suffices to get rid
of the last term $b\eta A_{22}$. Observe that $A_{22}\ge 0$.
Therefore in the case $\eta\ge 0$ we have $b\eta A_{22}\ge 0$ and the result follows.
It remains to consider the case $\eta\le 0$. Observe that
$
 A_{22}\le 4\sigma(x^2+y^2)
$,
therefore
\begin{equation}
\label{e-A22}
 |b\eta A_{22}| \le 4{|\eta|}\sigma(x^2+y^2) .
\end{equation}
In the case $|\eta|< c_1\sigma/10$, this implies that
$$
|b\eta A_{22}| \le \tfrac12c_1\sigma^2(x^2+y^2),
$$
so the term $b\eta A_{22}$ in \eqref{e-L0-almost} is majorized by
the term $c_1\sigma^2(x^2+y^2)$.
And in the case $|\eta|\ge c_1\sigma/10$, the fact
that $\eta\le 0\le \eta_0$
implies
$$
 c_0(\eta-\eta_0)^2 \ge c_0 \eta^2 \ge c_3\sigma^2
$$
where $c_3=c_0c_1^2/100$.
Recall that $|x|,|y|\ll\sigma$ (we are choosing $U$ after $\sigma$),
so we may assume that $x^2+y^2<c_3\sigma/10$. Then \eqref{e-A22} implies
that
$$
 |b\eta A_{22}| \le \tfrac12 c_3\sigma^2 \le \tfrac12 c_0(\eta-\eta_0)^2,
$$
so the term $b\eta A_{22}$ in \eqref{e-L0-almost} is majorized by
the term $c_0(\eta-\eta_0)^2$.

Thus we have proved \eqref{e-L0-goal}. This finishes
the proof of Lemma \ref{l-constant-metric}.
\end{proof}

\subsection{General case}
Since every metric is close to a constant one in
a neighborhood of the origin, Lemma \ref{l-constant-metric}
easily generalizes to arbitrary Finsler metrics on the plane.
Namely the following holds.

\begin{lemma}
\label{l-general-metric}
Let $\phi$ be a Finsler metric on $\R^2$.
Then there
exist a neighborhood $U\subset\R^2$ of the origin and a
smooth saddle embedding $F\colon U\to\R^4$ such that
the map $dF|_{S_\phi U}$ is an embedding and its image
$dF(S_\phi U)$ is a pre-convex surface in $\R^4$.
\end{lemma}

\begin{proof}
Let $\phi_0=\phi|_{T_0\R^2}$. We also consider $\phi_0$
as a constant Finsler metric on $\R^2$. For every $\ep>0$,
define a ``blow-up'' metric $\phi_\ep$ on $\R^2$ defined by
$$
 \phi_\ep(x,v) = \phi(\ep^{-1},v), \qquad x\in\R^2,\ v\in T_x\R^2.
$$
Note that $\phi_\ep$ converge to $\phi_0$ with all derivatives
on compacts sets as $\ep\to 0$.

By Lemma \ref{l-constant-metric}, there is a neighborhood $U\subset\R^2$
of the origin and a strictly saddle embedding $F\colon U\to\R^4$
such that $\Sigma_0:=dF(S_{\phi_0}U)$ is a pre-convex surface in $\R^4$.
Fix a neighborhood $U'\Subset U$ of the origin. Since the surfaces
$\Sigma_\ep:=dF(S_{\phi_\ep}U)$ converge to $\Sigma_0$ with all derivatives
on compact sets as $\ep\to 0$, the smaller surfaces $\Sigma'_\ep:=dF(S_{\phi_\ep}U')$
are pre-convex for all sufficiently small $\ep>0$. Fix such an $\ep$
and observe that the map
$$
 F_\ep: x \mapsto \ep^{-1} F(\ep^{-1}x),
$$
from the neighborhood $\ep U'$ of the origin to $\R^4$,
parameterizes a surface homothetic to $F$ in $\R^4$
(and hence is strictly saddle)
and $dF_\ep(S_\phi(\ep^{-1}U')) = \Sigma'_\ep$.
Thus $F_\ep$ and $\ep U'$ suit for $F$ and $U$
from the statement of the lemma.
\end{proof}

Now we are in position to prove Theorem \ref{t-saddle-embedding}.
Since the statement of the theorem is local, it suffices
to prove it for $M=(\R^2,\phi)$ and $x=0$ where $\phi$
is a Finsler metric on $\R^2$.
By Lemma \ref{l-general-metric}, there is
a neighborhood $U\subset\R^2$ of the origin and a
smooth saddle embedding $F\colon U\to\R^4$ such that
the map $dF|_{S_\phi U}$ parameterizes a pre-convex
hypersurface $\Sigma\subset\R^4$. Note that $\Sigma$
is symmetric with respect to the origin.

By Lemma \ref{l-convex-extension}, there exists a symmetric,
compact, smooth, quadratically convex surface $\Sigma'\subset\R^4$ 
which contains a neighborhood $U_0$ of the set
$K=dF(S_0)\subset\Sigma$ where $S_0$ is the unit sphere of $\phi$
in $T_0\R^2$. This surface is the unit sphere of some
smooth and quadratically convex norm $\|\cdot\|$ on $\R^4$.

For a sufficiently small neighbourhood $U'\subset U$ of 0,
we have $dF(S_\phi U')\subset U_0\subset\Sigma'$.
Therefore $\|dF(x,v)\|=1$ for every $x\in U'$
and every $\phi$-unit vector $v\in T_x\R^2$.
This means that $F$ is an isometric embedding
of $(U',\phi)$ to $(\R^4,\|\cdot\|)$.
This finishes the proof of Theorem \ref{t-saddle-embedding}.

\section{Complete convex surfaces}
\label{sec-convex}

The goal of this section is to prove Theorem \ref{t-convex-surface}.
Our plan is the following. Assuming that there is a geodesic line
on a surface of a convex set $B$ in a 3-dimensional normed space $V$,
we rescale $B$ with coefficients going to zero and pass to the limit.
This yields a geodesic line on the surface of the asymptotic cone of $B$,
and this geodesic line contains the cone's apex.
However on a surface of a sharp convex cone
no shortest path can pass through the apex, as shown in Lemma \ref{l-sharp-cone}.

A straightforward realization of this plan would require us to prove
that intrinsic metrics of converging convex surfaces converge to
the intrinsic metric of their limit (which is not necessarily smooth).
While this fact is standard in Euclidean spaces and
certainly true in general normed spaces,
we do not know an elegant proof and
do not want to mess with a cumbersome one here.
We work around this issue by constructing shortcut paths lying
in planar sections of our surfaces (and for planar convex curves
the convergence of lengths is easy, see Lemma \ref{l-convex-curves}).

\begin{notation}
For a vector space $V$ and points $p_1,p_2,\dots,p_n\in V$,
we denote by
$[p_1,p_2,\dots,p_n]$ the broken line
composed of segments $[p_ip_{i+1}]$, $i=1,\dots,i-1$.
If $V$ is equipped with a norm $\|\cdot\|$,
the length of this broken line is given by
$$
\len[p_1,p_2,\dots,p_n] = \sum_{i=1}^{n-1} \|p_i-p_{i+1}\| .
$$
\end{notation}

\medskip
We need the following standard fact
about perimeters of two-dimensional
convex sets
(supplied with a proof for the sake
of completeness).

\begin{lemma}
\label{l-convex-curves}
Let $V$ be a two-dimensional normed space
and $B\subset V$ a compact convex set
with nonempty interior. Then

1. For every compact convex set $B'\supset B$ one has
$\len(\pd B)\le\len(\pd B')$.

2. If $\{B_i\}$ is a sequence of convex sets in $V$
converging to $B$ (in the Hausdorff metric),
then $\len(\pd B_i)\to \len(\pd B)$.
\end{lemma}

\begin{proof}
1. Since the length of $\pd B$ is a
limit of lengths of inscribed polygons, it suffices
to prove the lemma in the case when $B$ is a polygon.
Let $\pd B=[p_1,p_2,\dots,p_n,p_1]$. If we cut $B'$
along a line containing a segment $[p_ip_{i+1}]$
and remove the piece that does not contain $B$,
the perimeter of $B'$ can only get smaller,
by the triangle inequality. Thus we can make $B$
from $B'$ by finitely many operations each of which
does not increase the perimeter.
Hence $\len(\pd B)\le\len(\pd B')$.

2. Choose the origin in the interior of $B$.
Then the assumption that $B_i\to B$ is equivalent
to the following:
$$
 (1-\ep_i)B \subset B_i \subset (1+\ep_i) B
$$
for some sequence $\ep_i\to0$. By the first part
of the lemma, this implies that
$$
 (1-\ep_i)\len(\pd B) \subset \len(\pd B_i) \subset (1+\ep_i) \len(\pd B),
$$
hence the result.
\end{proof}

\begin{lemma}
\label{l-sharp-cone}
Let $V$ be a 3-dimensional normed space whose norm
is $C^1$-smooth and strictly convex.
Let $K\subset V$ be a sharp cone.
%Then the intrinsic metric of $\pd K$
%has no geodesics going through the origin.
Then for every two points $p,q\in\pd K\setminus\{0\}$
there exists a path that connects $p$ and $q$ in $\pd K$,
is strictly shorter than the broken line $[p,0,q]$,
and is contained in some plane $\alpha\subset V$.
\end{lemma}

\begin{proof}
Let $H_1$ and $H_2$ be supporting planes to $K$
at $p$ and $q$ respectively. Since the cone is sharp,
there is a third supporting plane $H_3$ that does not
contain the intersection line $H_1\cap H_2$.
Consider the trihedral cone $K'=H_1^+\cap H_2^+\cap H_3^+$
where $H_i^+$ denotes the half-space bounded by $H_i$
and containing $K$.

It suffices to prove the lemma for $K'$ in place of $K$.
Indeed, suppose that for some plane $\alpha\subset V$
a boundary arc $\sigma'$
of $F':=\alpha\cap K'$ between $p$ and $q$ is shorter than $[p,0,q]$.
Consider the corresponding (that is, lying in the same half-plane
with respect to the line $\langle pq\rangle\subset\alpha$)
boundary arc $\sigma$ of $F:=\alpha\cap K$.
Since $F\subset F'$, Lemma \ref{l-convex-curves}
implies that
$$
\len(\sigma)\le\len(\sigma')<\len[p,0,q]
$$
and the lemma follows from its restatement for $K'$.

Thus now we restrict ourselves to
proving the assertion for $K'$.

Let $v$  be a nonzero vector in the line $H_1\cap H_2$
pointing outwards $K'$ (that is, $-v$ points in
the direction of an edge of $K'$).
Define
$$
 f(t) = \len[p,vt,q]=\|p-vt\|+\|q-vt\| .
$$
Note that $f$ is a strictly convex function differentiable at~0.

If $f'(0)>0$, then $f(-t)<f(0)$ for a small $t>0$.
Observe that $f(-t)$ is the length of the broken line
$[p,-vt,q]$ which lies on $\pd K'$ and is contained
in a plane (since it has only two edges).
Thus we have found a desired broken line
in the case when $f'(0)>0$.

It remains to consider the case when $f'(0)\le 0$.
For every $t\ge 0$, let $a(t)$ and $b(t)$ denote
the intersection points of segments
$[p,vt]$ and $[q,vt]$ with the plane $H_3$.
Note that $a(t)$ and $b(t)$ lie on edges of $K'$
and the broken line $[p,a(t),b(t),q]$ is contained
in $\pd K'$.
For $t=0$, we have $a(0)=b(0)=0$.

One easily sees that $a(t)$ and $b(t)$ are differentiable
in $t$ and their derivatives at~0 are nonzero
vectors (pointing in the directions of the respective edges). 
Denote these vectors by $v_1$ and $v_2$
and define
$$
 g(t) = \len[p,a(t),b(t),q] .
$$
Then
$$
 f(t) - g(t) = \|vt-a(t)\| + \|vt-b(t)\| - \|a(t)-b(t)\|.
$$
Therefore
$$
 \lim_{t\to+0} \frac{f(t)-g(t)}t = \|v-v_1\|+\|v-v_2\| - \|v_1-v_2\| > 0
$$
by the strict triangle inequality for the norm $\|\cdot\|$. Hence
$$
 \lim_{t\to +0}\frac{g(t)-f(0)}t
 = f'(0) - \lim_{t\to+0} \frac{f(t)-g(t)}t < 0
$$
since $f'(0)\le 0$. Therefore $g(t)<f(0)$
for all sufficiently small $t>0$.
Thus, for a small $t>0$, 
the broken line $[p,a(t),b(t),q]$
is shorter than $[p,0,q]$.
By construction, this broken line lies in the plane
through the points $p$, $q$ and $vt$.
\end{proof}

\begin{proof}[Proof of Theorem \ref{t-convex-surface}]
We may assume that the origin is contained in the interior of $B$.
Suppose that there is a geodesic line $\gamma\colon(-\infty,\infty)\to\pd B$.
For every $\lambda>1$, let $H^\lambda\colon V\to V$ denote the homothety
with coefficient $\lambda^{-1}$, that is, $H^\lambda(x)=\lambda^{-1}x$ for all $x\in V$.
Let $B^\lambda=H^\lambda(B)$ and $\gamma^\lambda\colon[-1,1]\to\pd B^\lambda$
is a path defined by $\gamma^\lambda(t)=H^\lambda(\lambda t))$.
Note that $\gamma^\lambda$ is a homothetic image of $\gamma|_{[-\lambda,\lambda]}$
reparameterized by arc length. 
Since $\gamma$ is a geodesic line on $\pd B$, $\gamma^\lambda$ is a shortest path
on $\pd B^\lambda$.

Now let $\lambda\to\infty$. The sets $B^\lambda$ converge to the asymptotic
cone $K:=\bigcap_{\lambda>1} B^\lambda$.
Since $B$ does not contain straight lines, $K$ is a sharp cone.
We assume that $K$ has nonempty
interior (the case when $K$ is contained in a plane is similar and
left to the reader).
Therefore the endpoints of the curves $\gamma^\lambda$ lie within a compact
region in $V$. Choose a subsequence $\{\lambda_i\}$,
$\lambda_i\to\infty$, such that $p_i:=\gamma^{\lambda_i}(-1)$
and $q_i:=\gamma^{\lambda_i}(1)$ converge to some points $p,q\in\pd K$.
Since the curves $\gamma^{\lambda_i}$ are 1-Lipschitz
and $\gamma^{\lambda_i}(0)=\lambda_i^{-1}\gamma(0)\to 0$,
the distances from $p$ and $q$ to the origin are not greater than~1.
Therefore by Lemma \ref{l-sharp-cone} there is a plane $\alpha\subset V$
such that a boundary arc $\sigma$ of $\alpha\cap K$ between $p$ and $q$
has $\len(\sigma)<2$.

We assume that $p\ne q$ (the case $p=q$ is trivial).
Fix a point $o\in\alpha\cap \operatorname{int}(K)$.
For each $i\gg 1$, let $\alpha_i\subset V$ be the
plane through $o$, $p_i$ and $q_i$.
Note that these planes converge to $\alpha$,
hence there are boundary arcs $\sigma_i$ of $\alpha_i\cap B^{\lambda_i}$
that converge to $\sigma$.
Consider a ``triangle'' $T\subset\alpha$ bounded by $\sigma$
and the segments $[op]$, $[oq]$. Applying Lemma \ref{l-convex-curves}
to $T$ and suitable projections of corresponding ``triangles''
in the planes $\alpha_i$ (and taking into account that the norms on $\alpha_i$
Lipschitz converge to the norm on $\alpha$) yields that $\len(\sigma_i)\to \len(\sigma)<2$.
Hence $\len(\sigma_i)<2=\len(\gamma^{\lambda_i})$ for a sufficiently large $i$.
Therefore $\gamma^{\lambda_i}$ is not a shortest path on $\pd B^{\lambda_i}$,
a contradiction.
\end{proof}

\bibliographystyle{plain}

\end{document}